\documentclass[12pt]{amsart}
\usepackage{amsfonts,amssymb,latexsym,amsmath, amsxtra}
\usepackage[all]{xy}
\usepackage[dvips]{graphics}

\usepackage[OT2,T1]{fontenc}
\DeclareSymbolFont{cyrletters}{OT2}{wncyr}{m}{n}
\DeclareMathSymbol{\Sha}{\mathalpha}{cyrletters}{"58}

\theoremstyle{plain}
\newtheorem{theorem}{Theorem}[section]

\newtheorem*{conjecture*}{Conjecture}

\theoremstyle{definition}

\theoremstyle{remark}
\newtheorem*{remark}{Remark}

\numberwithin{equation}{section}

\newcommand{\R}{\mathbb R}
\newcommand{\N}{\mathbb N}
\newcommand{\Z}{\mathbb Z}
\newcommand{\C}{\mathbb C}
\newcommand{\cS}{\mathcal{S}}

\newcommand{\Q}{{\mathbb Q}}

\def\cM{\mathcal{M}}

\def\cH{\mathcal{H}}

\def\H{\mathbb H}

\def\SL{\rm SL}

\def\sgn{\rm sgn}
\def\({\left(}
\def\){\right)}

\newcommand{\ol}[1]{\overline{{#1}}}

\newcommand{\abs}[1]{\left|#1\right|}

\newcommand{\smatr}[4]{\(\begin{smallmatrix} #1 & #2 \\ #3 & #4\end{smallmatrix}\)}

\def\k2{\frac{k}{2}}

\begin{document}

\title[A Maass Lift of $\Theta^3$]{A Maass Lifting of $\Theta^3$ and \\Class Numbers of Real and imaginary Quadratic Fields}

\author{Robert C. Rhoades}
\address{Stanford University\\ Department of Mathematics\\ Bldg 380 \\ Stanford, CA 94305 }
\email{rhoades@math.stanford.edu}

\author{Matthias Waldherr}
\address{Mathematical Institute\\University of
Cologne\\ Weyertal 86-90 \\ 50931 Cologne \\Germany}
\email{mwaldher@math.uni-koeln.de}

\thanks{The first author was partially supported by an NSF Postdoctoral Fellowship and the Chair in Analytic Number Theory at EPFL}
\subjclass[2000] { }

\date{\today}
\thispagestyle{empty} \vspace{.5cm}

\begin{abstract}
We give an explicit construct of a harmonic weak Maass form
$F_{\Theta}$ that is a ``lift''  of $\Theta^3$, where $\Theta$ is
the classical Jacobi theta function. Just as the Fourier
coefficients of $\Theta^3$ are related to class numbers of imaginary
quadratic fields,
 the Fourier coefficients of the ``holomorphic part'' of $F_{\Theta}$ are associated to class numbers  of
real quadratic fields.
\end{abstract}

\maketitle

\section{Introduction and Statement of Results}\label{sec:Introduction}
Ramanujan's mock theta functions proved mysterious for
more than 80 years.  They are $q$-hypergeometric series such as
\begin{equation}\label{eqn:f}
f(\tau):=1+\sum_{n=1}^{\infty}\frac{q^{n^2}}{(1+q)^2(1+q^2)^2\cdots (1+q^n)^2},
\end{equation}
with $q:=e^{2 \pi i \tau}$ ($ \tau \in \H$),
 that
have ``nearly modular'' properties, but fail to be fully modular.
Modularity, explained by Zwegers in his Ph.D. thesis \cite{zwegers1, Zw}, is obtained if one adds  to the mock theta function  a certain non-holomorphic integral
$$P_g(\tau):=
\int_{-\overline{\tau}}^{i\infty}
\frac{g(z)}{
\sqrt{-i(\tau+z)}} \ dz.
$$
Here $g$ is a weight $3/2$ unary theta function, which in the case of the mock theta function $f$ is given  by
$$
g(\tau):=\frac{1}{6} \sum_{n \equiv 1 \pmod{6}} n
q^{\frac{n^2}{24}}
$$%
(see \cite{zwegers1}).
The resulting function, $\cM_f(\tau) := q^{-\frac{1}{24}}f(\tau)+2i\sqrt{3}P_g(\tau)$, is a harmonic weak Maass form
of weight $1/2$ (see Section \ref{sec:Poincare}
for the definition). Given $\cM_f$ we may recover $g$ by applying a differential operator $\xi_{\frac12}$
(also see Section \ref{sec:Poincare}).
Following Zagier \cite{Za3}, we call the image of $\cM_f$ under $\xi_{\frac12}$
the \textit{shadow} of the mock theta function $f$. In this case
$\xi_{\frac12} \cM_f$ is a unary theta function. In general, the shadow is a modular form. In view of this we define a
\emph{mock modular form} to be the holomorphic part (see Section \ref{sec:Poincare})
of a harmonic weak Maass form.


Conversely, one may begin with a non-holomorphic integral of  a weight $3/2$
unary theta function and produce a mock theta function.
The resulting mock theta functions may be written as Lerch sums \cite{Zw} and in many cases may also be written
as $q$-series similar to that in \eqref{eqn:f}.
A similar construction exists for the analogous non-holomorphic integral of a
weight $1/2$ unary theta function \cite{BFO, BL}.
However, no such construction exist for nonunary theta functions.

This raises the question whether a nonunary theta function may appear as the shadow of a mock modular form.  
Let $F$ be a hamonic weak Maass form of weight $1/2$.
In this paper, we give an explicit construction of such an example, namely $\Theta^3$, where
$$\Theta(\tau)  := \sum_{n\in \Z} q^{n^2}.$$
By work of Bruinier and Funke \cite{BF} we know the existence of
such a mock modular form,
but not its explicit form.  In fact, \cite{BF} implies the existence of a mock modular form with shadow
equal to any holomorphic modular form of positive weight.

In recent work  Duke, Imamo\={g}lu, and T\'{o}th \cite{DIT}
construct mock modular forms that have certain weight $3/2$ weakly holomorphic forms as
their shadows. Their work and work of Knopp \cite{Kn1}
suggests that such a mock modular form can be
constructed from a weight $1/2$ non-holomorphic Poincar\'e series.
Earlier work of Kubota
\cite{K} also gives insight into the construction of such a form.
Along these lines, we give an explicit construction of a harmonic
weak Maass form with shadow equal to $\Theta^3$ and prove that the
coefficients of the associated harmonic weak Maass form are related
to class numbers of real and imaginary quadratic fields.

The theory of mock modular forms has exploded in recent years and
with its development have come many questions concerning the
arithmetic nature of the Fourier coefficients of mock modular forms
(see, for example, \cite{BrO, bruinierrhoades, DIT, ZwIndef}).
For instance, one might ask: When are the coefficients rational or
algebraic? It is believed that the Fourier coefficients of a mock
modular form are rational when the shadow is a modular form with
complex multiplication. For instance, as a result of the
constructions for unary theta functions it is clear that the
coefficients of the associated mock modular form are integral (see, for example,
\cite{ZwGrid}).

This paper demonstrates that the coefficients of lifts of nonunary
theta function are essentially given by special values of Dirichlet
$L$-functions associated to quadratic fields. In our specific case
the coefficients of the mock modular form are related to the
logarithms of fundamental units of quadratic fields.
Despite not being integral, the
coefficients of the mock modular forms associated with theta
functions still carry arithmetic data. For more examples of mock
modular forms whose coefficients are not rational but still encode
arithmetic information see, for instance, \cite{BrO, bruinierrhoades,
DIT}.

Returning to our example, let us start by recalling what is known
about the relationship between $\Theta^3$ and $L$-functions of
quadratic fields. Recall that the Fourier coefficients of
$\Theta^3$, which we denote by $r(n)$, themselves encode class
numbers. To be more precise,  for $N>0$ and $N \equiv 0,3 \pmod 4$,
we write $H(-N)$ for the Hurwitz class number, i.e., the number of
equivalence classes of quadratic forms of discriminant $-N$, where
each class $C$ is counted with multiplicity $1/\text{Aut}(C)$. Then
  $$
  r(n) =
   \begin{cases}12H(-4n) & \text{ if } n\equiv 1, 2\pmod{4},
    \\ r(n/4) & \text{ if } n\equiv 0 \pmod{4},
    \\ 24 H(-n) & \text{ if } n\equiv 3 \pmod{8},
    \\ 0 & \text{ if } n\equiv 7 \pmod{8}.
    \end{cases}
    $$
The Hurwitz class number $H(-N)$ itself is related to the class
number $h(-N)$ of the ring of integers of  $\Q(\sqrt{-N})$. To state
this relationship accurately, write $-N=- \Delta_Nf^2$, where $-
\Delta_N$ is a negative fundamental discriminant. Then we have that
\cite{Co}
$$
H(-N)= \frac{2 h(-\Delta_N)}{\omega_N}
T_1^{\psi_{-N}}(f),
$$
where $\omega_N$ denotes the number of units in $\Q(\sqrt{-N})$ and for $n \not=0$, $\psi_{n}(\cdot):=\left(\frac{D}{\cdot} \right)$
with $D$ the discriminant of $\Q(\sqrt{n})$. Moreover,
 for a character $\chi$, $T_s^\chi$ is the multiplicative function defined by
\begin{align*}
T^{\chi}_s(w) :=
\sum_{a|w} \mu(a)\chi(a)a^{s-1}\sigma_{2s-1}\left(\frac{w}{a}\right),
\end{align*}
where $\mu$ is the M\"obius function and $\sigma_{\ell}$ denotes the  $\ell$th divisor sum.
We also write $T_1:=T_1^{\psi_1}$.

Finally the relationship between $r(n)$ and special values of
$L$-functions of quadratic fields is given by Dirichlet's class
number formula, which for $\psi_n$ with $n<0$ states
$$ L(1, \psi_n) = \frac{2\pi}{w_n\sqrt{-n}} h(n).
$$
Here, $L(s,\chi)$ is the Dirichlet $L$-function associated to
character $\chi$.

As a conclusion we may note that the Fourier coefficients $r(n)$ are
related to the value of Dirichlet $L$-functions associated to
imaginary quadratic fields at $s=1$. In our work we show that the
complementary values for real quadratic fields appear as
coefficients in the mock modular form having shadow $\Theta^3$. To
state our results accurately define $F_{\Theta}:\H \to \C$ by the
following Fourier expansion
\begin{equation*}
F_{\Theta}(\tau):= \sum_{n = 0}^\infty c^+(n)q^n+
2y^\frac{1}{2}+ \sum_{n =1 }^{\infty} c^-(n)
\Gamma(\tfrac{1}{2};4 \pi n y)q^{-n},
\end{equation*}
where $\Gamma(a;x):=\int_{x}^{\infty} e^{-t} t^{a-1} dt$ denotes the incomplete gamma-function,
and the coefficients $c^+(n)$ and $c^-(n)$ are given by
\begin{align*}
 c^+(n):=& \pi  e^{-\frac{\pi i}{4}}
 \overline{Z_{-n}},\\
 c^-(n):=&\sqrt{\pi}   e^{-\frac{\pi i}{4}}
 \overline{Z_{n}}.
 \end{align*}
\begin{remark}
In Section \ref{sec:Poincare} it is shown that
$$c^-(n) = -\frac{1}{2\sqrt{\pi n}} r(n).$$
\end{remark}
 To define the values   $Z_n$, we write  $n \neq 0$ as $n=f^2d$ with $d$ squarefree
and $f=2^qw$ with $w$ odd, and let
 $$
c_{n}:=
\left\{
\begin{array}{ll}
2-\psi_{-n}(2) & \hbox{if } n \equiv 1,2 \pmod{4},\\
   2^{-Q} \left(1-\psi_{-n}(2) \right) & \hbox{otherwise},
\end{array}%
\right.
$$
where
$$
Q:=
\begin{cases}
q&\text{if } d \equiv 3 \pmod 4,\\
q-1&\text{if } d \equiv 1,2 \pmod 4.
\end{cases}
$$
 Then we define
 $$
 Z_n:= \left\{
 \begin{array}{ll}
 e^{\frac{3\pi i }{4}} \frac{6}{\pi^2} \log (2)&\text{if } n=0, \\
 e^{\frac{3\pi i }{4}} \frac{6}{\pi^2} \log (2) \frac{T_1(w)}{w}&\text{if } n \text{ is a square},\\
 e^{\frac{3\pi i}{4}}\frac{6}{\pi^2}L(1,\psi_{-n}) \frac{T^{\psi_{-n}}_1(w)}{w} \cdot c_{n}&\text{otherwise}.
 \end{array}
 \right.
 $$
Let $$F_{\Theta}^+(\tau) := \sum_{n=0}^\infty c^+(n)q^n.$$
Dirichlet's class number formula for real quadratic fields, that is for $n>0$, states
 $$
 L(1, \psi_n) = \frac{\log(\epsilon_n)}{\sqrt{n}} h(n),
$$
where $\epsilon_n$ is the fundamental unit in the field $\Q(\sqrt{n})$.  Therefore,
the coefficients $F_{\Theta}^+$ may be written as simple expressions in terms of class numbers.

\begin{theorem}\label{thm:cubelift}
The function
$F_{\Theta}^+$ is a mock modular form of weight
$\frac{1}{2}$ with respect to $\Gamma_0(4)$ with shadow $\Theta^3$.
Furthermore, the harmonic weak Maass form $F_{\Theta}$ is a Hecke eigenform.
\end{theorem}

\begin{remark}
Nonunary theta functions are closely related to Eisenstein series of half integral weight.
Such series typically have Whittaker-Fourier coefficients equal to a quotient of Hecke $L$-functions, often associated to
imaginary quadratic fields. In general, a
mock modular form with shadow equal to a nonunary theta function will have Fourier coefficients of the same shape,
often associated to real quadratic fields.

See Section \ref{sec:DIT} for further discussion of the work of
Duke, Imamo\=glu, T\'oth and other works dealing with the arithmetic nature of the Fourier
coefficients of harmonic weak Maass forms.
\end{remark}

In Section \ref{sec:Poincare} we construct a Maass-Poincar\'e series of weight $1/2$ related to $\Theta^3$.
In Section \ref{sec:Fourier} we compute its Fourier expansion, resulting in the relations to $L$-series and proving Theorem  \ref{thm:cubelift}.

\section*{Acknowledgements}
The authors thank Kathrin Bringmann for valuable discussions and guidance.
The authors also thank the referee for comments that helped polish the exposition of the paper.

\section{A Maass-Poincar\'e series representation for $F_{\Theta}$}\label{sec:Poincare}

In this section we write  $F_{\Theta}$ as a Poincar\'e series.
We begin by recalling the definition of a harmonic weak Maass form.
With $\Gamma$ a finite index subgroup of $\SL_2(\Z)$ and
$\nu:\Gamma \to \C$ a multiplier,
a \textit{harmonic weak Maass form} of weight $k$ with respect to $\Gamma$ is a smooth
function $F:\H \to \C$ with the following properties
\begin{enumerate}
\item For all $A = \smatr{a}{b}{c}{d} \in \Gamma$ we have $F(A \tau)= \nu(A)(c \tau+d)^{k} F(\tau)$.
\item We have that $\Delta_k F =0$, where
 for  $z=x+iy$ with $x, y\in \R$,  the weight $k$ \textit{hyperbolic
Laplacian} is given by
\begin{equation}\label{laplacian}
\Delta_k := -y^2\left( \frac{\partial^2}{\partial x^2} +
\frac{\partial^2}{\partial y^2}\right) + iky\left(
\frac{\partial}{\partial x}+i \frac{\partial}{\partial y}\right).
\end{equation}
\item $F$ has at most linear exponential growth toward each cusp of $\Gamma\backslash \H$.
\end{enumerate}
The \emph{shadow} of a harmonic weak Maass form $F$ is a weakly holomorphic modular form of weight
$2-k$ equal to $\xi_{k}(F)$ where
$\xi_{k}:=2i y^{k}\overline{\tfrac{\partial}{\partial \bar \tau} }$.

Every weight $k$ harmonic weak Maass form $F(z)$
has a Fourier expansion of the form
\begin{equation}\label{fourier}
F(\tau)=\sum_{n\gg -\infty} c_F^+(n) q^n + Cy^{1-k}+ \sum_{n\ll +\infty, n\ne 0} c_F^-(n)
\Gamma(1-k, -4\pi ny) q^n,
\end{equation}
As \eqref{fourier} reveals, $F(z)$
naturally decomposes into two summands
\begin{eqnarray}
\label{fourierhh}
&F^{+}(\tau):=\sum_{n\gg -\infty} c_F^+(n) q^n,\\
\label{fouriernh}
&F^{-}(\tau):=Cy^{1-k} + \sum_{\substack{n\ll +\infty, n\ne 0 }} c_F^-(n)\Gamma(1-k, -4\pi ny)q^n.
\end{eqnarray}
A direct computation shows that
$\xi_{k}(F)$
is given simply in terms of $F^{-}(z)$, the \emph{non-holomorphic part}
of $F$.
The \emph{holomorphic part} of $F$ is $F^{+}(z)$.

Next we recall a series representation for $r(n)$. For this, we  define for $\left(\begin{smallmatrix}a&b \\ c&d \end{smallmatrix} \right) \in \Gamma_{\Theta} := \left< \smatr{1}{2}{0}{1}, \smatr{0}{1}{-1}{0}\right>$ the theta multiplier \cite{Kn2}
\[
 \nu_\Theta\left(%
\(
\begin{smallmatrix}
  a & b \\
  c & d \\
\end{smallmatrix}%
\)
\right) :=
\left\{%
\begin{array}{ll}
   \left( \frac{d}{c} \right)^* e^{-\frac{\pi i c}{4}} & \hbox{if } b \equiv c \equiv 1 \pmod{2}, a \equiv d \equiv 0 \pmod{2},\\[1ex]
   \left( \frac{c}{d} \right)_* e^{\frac{\pi i (d-1)}{4}} & \hbox{if } b \equiv c \equiv 0 \pmod{2},   a \equiv d \equiv 1 \pmod{2}.\\
\end{array}%
\right.   \]
Here, for $c \not=0$, we define, using the usual Jacobi symbol,
\begin{eqnarray*}
\left( \frac{c}{d} \right)^*&:=& \left( \frac{c}{|d|} \right),\\
 \left( \frac{c}{d} \right)_*&:=& \left( \frac{c}{|d|} \right)(-1)^{\frac{\sgn (c) -1}{2}\frac{\sgn (d) -1}{2}}.
\end{eqnarray*}
Moreover, we set
$$
\left( \frac{0}{\pm 1} \right)^* =  \left( \frac{0}{1} \right)_*= -\left( \frac{0}{-1} \right)_*=1.
$$
\begin{remark}
$\nu_\Theta$  is the multiplier for $\Theta(\tau/2)$, a form on $\Gamma_\Theta$, rather than on $\Gamma_0(4)$.
\end{remark}
For $c \in \N$, define the sum of Kloosterman type
\begin{equation} \label{Klooster}
 \mathcal{S}(n;c) :=
 \sum_{d \pmod{2c}}
\overline{\lambda(d,c)}^3 e^\frac{\pi i d n}{c}
\end{equation}
with
$$
\lambda(d,c):=
 \left\{
\begin{array}{ll}
   e^{-\frac{\pi i c}{4}}\left(\frac{d}{c}\right) & \hbox{if } c \hbox { is odd, } d \hbox { is even,} \\
   e^\frac{\pi i (d-1)}{4}\left(\frac{c}{d}\right)  & \hbox{if } c \hbox { is even, } d \hbox { is odd,} \\
   0 & \hbox{otherwise.} \\
\end{array}%
\right.$$
We require the  Kloosterman zeta-function, which is defined for Re$(s)$ sufficiently large,
$$
Z_n(s):=\sum_{c=1}^\infty
\frac{\mathcal{S}(n;c)}{c^{s+\frac{1}{2}}}.
$$
It is known (see for example \cite{Kn1}) that for $n\not=0$, $Z_n(s)$ has an analytic continuation to $s=1$. The analytic continuation of $Z_0(s)$ to $s=1$ is shown in Theorem \ref{thm_E0}.
Using the above notation, we can state the following series expansion for $r(n)$ (a proof may for example be found in \cite{Ba})
\begin{equation} \label{rn}
r(n)
= 2 e^{ -\frac{3 \pi i }{4}}  \pi n^{\frac12 }
Z_n(1).
\end{equation}

To construct the Poincar\'e series required, we let
$\psi(\tau;s):= 2^{-s+  \frac14} y^{s-\frac{1}{4}}$ and    $\Gamma_\infty(2) := \left\{ \pm \smatr{1}{2n}{0}{1};n \in \Z\right\}$.
We formally define the Poincar\'e series
\begin{equation*}
F_{\Theta}(\tau;s):
= \sum_{A \in \Gamma_\infty(2) \setminus \Gamma_\Theta} \psi(A\tau;s)\nu^{3}_\Theta(A) (c\tau+d)^{-\frac{1}{2}}.
\end{equation*}
One can show that for $s>3/4$  the function $F_{\Theta}(\tau;s)$ is absolutely convergent and  transforms like an automorphic form of weight $\frac12$ with multiplier
$\nu^{-3}_\Theta$ and eigenvalue $\left( s-\frac{1}{4} \right)\left( \frac{3}{4}-s \right)$ under the weight $\frac12$ hyperbolic Laplacian
$
\Delta_{\frac12}
.
$
We are  interested in the case $s=\frac34$ which will be obtained by continuing the Fourier expansion
of
$F_{\Theta}(\tau;s)$ analytically.

To state the Fourier expansion of $F_{\Theta}(\tau;s)$, we define
\[
\mathcal{W}_n(y;s):=
\left\{
\begin{array}{ll}
  |n|^{-\frac{1}{2}} \Gamma\left(s+\frac{{\sgn}( n)}{4}\right)^{-1} \left(4 \pi |n|y\right)^{-\frac{1}{4}}
  W_{\frac{1}{4}{\sgn}(n), s-\frac{1}{2}}
  \left(4 \pi \abs{n}y\right) & \hbox{if } n \neq 0, \\
  \frac{2^{2s-\frac{1}{2}}}{(2s-1)\Gamma\left(2s-\frac{1}{2}\right)}y^{\frac{3}{4}-s} & \hbox{if } n=0, \\
\end{array}
\right.
\]
where $W_{\nu,\mu}$ is the usual $W$-Whittaker function.

\begin{theorem}\label{thm_fourierexpansion}
We have the following Fourier expansion
\[
F_{\Theta} (\tau;s) = \left(\frac{y}{2} \right)^{s-\frac{1}{4}} + \sum_{n \in \Z} a_n(s) \mathcal{W}_n\left(\tfrac{y}{2};s\right) e^{ \pi i nx},
\]
where
\[
a_0(s)=  2^{1-4s} e^{-\frac{\pi i}{4}}
\pi^{\frac{1}{2}} \Gamma(2s)
\overline{  Z_0\left(2 \bar{s}-\frac12 \right)  }
\]
and for  $n \neq 0$
\[
a_n(s) = 2^{\frac12-2s}\pi^{s+\frac{1}{4}} |n|^{s-\frac{1}{4}}
e^{-\frac{\pi i}{4}} \overline{  Z_{-n}\left(2 \bar{s}-\frac12 \right)  }
.
\]
Moreover, the series $F_{\Theta} (\tau;s)$ has an analytic continuation to $s= \frac34$ and we have the expansion
\begin{align*}
F_{\Theta} (\tau):=& 2 F_{\Theta}\left(2\tau;\frac{3}{4}\right) \\
=&
 2 y^\frac{1}{2}
 +\frac12  \pi e^{-\frac{\pi i }{4}} \ol{Z_0(1)}
+
 e^{-\frac{\pi i}{4}} \pi  \sum_{n = 1}^\infty
 \overline{Z_{-n}(1)} q^n
 +  e^{-\frac{\pi i}{4}}\sqrt{\pi} \sum_{n
=1 }^{\infty}
 \overline{Z_{n} (1)} \Gamma\left(\tfrac{1}{2};4 \pi n y\right)q^{-n}.
\end{align*}
The function $F_{\Theta}$ is a harmonic weak Maass form of weight $\frac12$ for $\Gamma_0(4)$   satisfying
\begin{equation} \label{shadow}
\xi_\frac{1}{2}(F_{\Theta})= \Theta^3.
\end{equation}

\end{theorem}
\begin{proof}
Since the proof of the Fourier expansion is quite standard   (see \cite{Fa} for a similar calculation), we do not give it here.
The analytic continuation of $F_{\Theta} (\tau;s)$ to $s=\frac34$ follows directly from the analytic continuation of $Z_n(2s-\frac12)$.
The expansion of $F_{\Theta}$ is then obtained by setting $s=\frac34$ and using special values of Whittaker functions (see \cite{DIT}, for example).
Moreover, it is well known that if $f(\tau)$ transforms like a modular form of weight $\frac12$ for $\Gamma_{\Theta}$ with multiplier $\nu_{\Theta}^{-3}$, then $f(2 \tau)$ transforms like a modular form
on $\Gamma_0(4)$.

Finally (\ref{shadow}) follows by a direct calculation, using the explicit form of $r(n)$ stated in (\ref{rn}).
More precisely, we have
$\xi_{\frac{1}{2}}\left(2y^{\frac{1}{2}}\right)=1$
and
$\xi_{\frac{1}{2}}\left( \Gamma\left(\frac12;y \right)  \right)=e^{-y}$.
Using the anti-linearity of $\xi_{\frac{1}{2}}$ then easily gives the claim.%
 \end{proof}

 We conclude this section by showing that  $F_{\Theta}$ is a Hecke eigenform.
  Since  $\Theta^3$ is a Hecke eigenform  with
 eigenvalue $1+p$ under the Hecke operator $T(p^2)$ (see \cite{Sh}), one may easily conclude that
 $$
F_{\Theta}|T(p^2) - \left( 1+ \frac{1}{p}\right) F_{\Theta}
$$
is a weakly holomorphic modular form of weight $\frac12$ on $\Gamma_0(4)$.
Moreover, its principal part is constant so it is a holomorphic modular form.
By the Serre-Stark basis theorem
 the space of holomorphic modular forms of on $\Gamma_0(4)$ is known to be one-dimensional
and spanned by $\Theta$.
Computing the action of the Hecke
operators explicitly, one sees that its constant term is $0$, thus the form must be 0.

\section{Relation to L-series}\label{sec:Fourier}
 In this section, we will show that $Z_n(1)=Z_n$, where $Z_n$ was defined in the introduction.
 For this, we will distinguish the cases $n \not=0$ and $n=0$.

\subsection{Computation of $Z_n(s)$ for $n\ne 0$}

\begin{theorem}\label{thm:EnEvenOdd}
Let $n = f^2 d\ne 0$ be an integer with $d$ square-free  and $f= 2^qw$ with $w$ odd.
Then we have that
 $$
 Z_n(s) = Z_n^{odd}(s) R_n(s)
 $$
with
$$
Z_n^{odd}(s) := e^{\frac{3\pi i}{4}} \frac{L(s, \psi_{-n})}{\zeta(2s)} w^{1-2s} T_s^{\psi_{-n}}(w) \frac{1- \psi_{-n}(2)2^{-s}}{1-2^{-2s}}
$$
and
$$
R_n(s):= 1+2^{-s}-2^{1-s}R_n^*(s).
$$
Here
$$
R_n^*(s):=
\begin{cases}
0 & \quad\text{if } n\equiv 1,2\pmod{4},\\
\frac{1-2^{-2s}}{1-\psi_{-n}(2)2^{-s}}2^{Q(1-2s)}T_s^{\psi_{-n}}(2^Q)  & \quad\text{otherwise}.
\end{cases}
$$

\end{theorem}
\begin{proof}
We first relate our functions to certain functions studied by Zagier \cite{Za1}.
For this define for $n\in\Z$
\[
\gamma_c(n):=\frac{1}{\sqrt{c}}\sum_{d=1}^{2c}\lambda_Z(d, c)\,e^{-\frac{\pi idn}{c}},
\]
where
\[
\lambda_Z(d, c):=
\begin{cases}
i^{\frac{1-c}{2}}\left(\frac{d}{c}\right) & \quad\text{if } c \text{ is odd, } d\text{ is even},\\
i^{\frac{d}{2}}\left(\frac{c}{d}\right)   & \quad\text{if } c \text{ is even, } d\text{ is odd},\\
0                                                              & \quad\text{otherwise.}
\end{cases}
\]
It is not hard to see that
\[
\overline{\lambda^{3}(d, c)}=e^{\frac{3\pi i}{4}}(-1)^{c+1} \lambda_Z (d, c)
\]
yielding
\[
S(n; c)=e^{\frac{3\pi i}{4}} (-1)^{c+1} \sqrt{c}\gamma_c(-n).
\]
We next split $Z_n(s)$ into an even and into an odd part of $c$.
For this, we write $c=2^r c'$ with $c'$ odd, $r\in\N$ and by \cite{Za1} for $r \geq1$ and  $N \not=0$  we may decompose $\gamma_c(N)$ as
\[
\gamma_c(N)=Q_r(N)\gamma_{c'}(N),
\]
where
\[
Q_r(N):=
\begin{cases}
2^{\frac{r}{2}}(-1)^{\frac{m-1}{4}}      & \quad\text{if } r \text{ is even, } N=2^{r-2}m,\, m\equiv 1\pmod{4},\\
2^{\frac{r-1}{2}}(-1)^{\frac{m(m-1)}{2}}  & \quad\text{if } r \text{ is odd, } N=2^{r-1}m,\\
0                                                              & \quad\text{otherwise.}
\end{cases}
\]
This gives
\[
Z_n(s)=e^{\frac{3\pi i}{4}}\sum_{c=1}^\infty\frac{(-1)^{c+1} \gamma_c(-n)}{c^s}=e^{\frac{3\pi i}{4}}\sum\limits_{c'=1\atop{c'\text{ odd}}}^\infty \frac{\gamma_{c'}(-n)}{c'^s}\left(1-\sum_{r=1}^\infty\frac{Q_r(-n)}{2^{rs}}\right)\ .
\]
By \cite{Za1}, we know that
\begin{equation*}
\sum\limits_{c'=1\atop{c'\text{ odd}}}^\infty\frac{\gamma_{c'}(-n)}{c'^s}
 =\prod_{p\neq 2}\frac{1-p^{-2s}}{1-\psi_{-n}(p)p^{-s}}w^{1-2s}T_s^{\psi_{-n}}(w)
 = \frac{1-\psi_{-n}(2)2^{-s}}{1-2^{-2s}}\frac{L(s, \psi_{-n})}{\zeta(2s)}w^{1-2s}T_s^{\psi_{-n}}(w).
\end{equation*}
To evaluate the second factor, we  define
\[
\widetilde{R}_N(s):=\frac{1}{2}\left(1+\sum_{r=1}^\infty\frac{Q_r(N)}{(2^{r-1})^s}\right)\ .
\]
In \cite{Za1} it is shown that
\[
\widetilde{R}_N(s)=
\begin{cases}
0 & \quad\text{if } N\equiv 2, 3\pmod{4},\\
\frac{1-2^{-2s}}{1-\psi_N(2)2^{-s}}2^{Q(1-2s)}T_s^{\psi_N}(2^Q)  & \quad\text{if } N=F^2D.
\end{cases}
\]
Here $D$ is the discriminant of $\Q(\sqrt{N})$ and we write $F=2^Qr$ with $r$ odd.
Now the claim follows from
\[
1-\sum_{r=1}^\infty\frac{Q_r(-n)}{2^{rs}}=
-2^{1-s} \widetilde{R}_{-n}(s)+2^{-s}+1.
\]
\end{proof}
To finish the evaluation of $Z_n(1)$, we distinguish whether $-n$ is a  square or not. If $-n$ is not a square,
 $L(s, \psi_{n})$ converges and we may deduce that $Z_n(s)$ converges for $s=1$.
 We may then evaluate $Z_n(s)$ at $s=1$
by using Theorem \ref{thm:EnEvenOdd}.
If $-n$ is a square, then $L(s,\psi_{-n})= \zeta(s)$ and we have
by Theorem \ref{thm:EnEvenOdd} that
$$
Z_n(1) = \frac{e^{\frac{3 \pi i}{4}}}{\zeta(2)}  \frac{T_1^{\psi_1}(w)}{w}
\lim_{s \to1} \frac{\zeta(s)R_n(s)}{\left(1+2^{-s}\right)} .
$$
One easily computes
$$
 \frac{R_n(s)}{\left(1+2^{-s}\right)}=1-2^{(1-s)}.
$$
Using that $\zeta(s)= \frac{1}{s-1}+O(1)$ as $s\to 1$, gives that
$$
\lim_{s \to1} \frac{\zeta(s)R_n(s)}{\left(1+2^{-s}\right)}
= \frac{d}{ds} \left. \left(1-2^{(1-s)} \right) \right|_{s=1}
= \log(2).
$$
From this we may conclude that $Z_n(1)=Z_n$.

\subsection{Computation of $Z_0(1)$} \label{n=0}

This subsection is devoted to the computation of $Z_0(s)$. 

\begin{theorem}\label{thm_E0} We have for $s>1$
\[
Z_0(s)= e^{\frac{3\pi i }{4}} \frac{\zeta(2s-1)}{\zeta(2s)}
\frac{1-2^{-(2s-1)}-2^{-s}}{1-2^{-2s}}.
\]
In particular $Z_0(s)$ has an analytic continuation to $s=1$.
\end{theorem}
\begin{proof}
We first assume that $c$ is odd. Then
\[
S(0; c)=\sum\limits_{d\pmod{2c}\atop{d\text{ even}}}\overline{\lambda(d, c)}^3=e^{\frac{3\pi ic}{4}}\sum\limits_{d\pmod{2c}\atop{d\text{ even}}}
\left(\frac{d}{c}\right)=\left(\frac{2}{c}\right)e^{\frac{3\pi ic}{4}}\sum_{d\pmod{c}}\left(\frac{d}{c}\right)\ .
\]
The last sum vanishes unless $c$ is  a square in which case it equals $\phi(c)$, thus in this case
\[
S(0; c)=e^{\frac{3\pi i}{4}} \phi(c).
\]
Next we assume that $c$ is even. Then
\[
S(0; c)=\sum\limits_{d\pmod{2c}\atop{d\text{ odd}}}\overline{\lambda(d, c)}^3=\sum\limits_{d\pmod{2c}\atop{d\text{ odd}}}e^{\frac{\pi i(d-1)}{4}}\left(\frac{c}{d}\right).
\]

\noindent
We write $c=2^rc'$ with $r\geq 1, c' $ odd,  and $d=d_1+2^{r+1}d_2$, where $d_1$ runs   $\pmod{2^{r+1}}$ and $d_2$ runs $\pmod{c'}$.
 Then
\[
e^{\frac{\pi i}{4}(d-1)}\left(\frac{c}{d}\right)=e^{\frac{\pi i}{4}(d_1-1)}\left(\frac{2}{d_1}\right)^r\left(\frac{c'}{d_1+2^{r+1}d_2}\right),
\]
therefore
\[
S(0; c)=\sum\limits_{d_1\pmod{2^{r+1}} }e^{\frac{\pi i}{4}(d_1-1)}\left(\frac{2}{d_1}\right)^r\sum_{d_2\pmod{c'}}
\left(\frac{c'}{d_1+2^{r+1}d_2}\right)\ .
\]
The sum on $d_2$ can be simplified as
\[
\sum_{d_2\pmod{c'}}
\left(\frac{c'}{d_2}\right)=
\begin{cases}
\phi(c')        & \quad\text{if } c' \text{ is a square},\\
0                       & \quad\text{otherwise.}
\end{cases}
\]
We next consider the sum on $d_1$. Changing $d_1\mapsto d_1+2$ one sees that this sum vanishes if $r$ is even. If $r$ is odd, then the sum easily evaluates to $-2^{r-\frac{1}{2}} e^{\frac{3\pi i}{4}}$.
%
%
We deduce that
\[
\mathcal{S}(0;c)=\begin{cases}
    e^{\frac{3\pi i
}{4}} \phi(c) & \hbox{if $r=0$ and $c'$ is a square}, \\
 -2^{r-\frac12}e^{\frac{3\pi i}{4}}\phi(c') & \hbox{if $c'$ is a square and } r \hbox{ is odd},\\
    0 & \hbox{otherwise}.
\end{cases}
\]
Combining the above gives that
$$
Z_0(s) =
\sum_{c \text{ odd}} \frac{\cS(0;c)}{c^{s+\frac{1}{2}}}
+ \sum_{r\ge 1} \sum_{c \text{ odd}} \frac{\cS\left(0;2^r c\right)}{(2^rc)^{s+\frac{1}{2}}}
=  e^{\frac{3\pi i}{4}} \sum_{ c \text{ odd}  }
\frac{\phi\left(c^2\right)}{c^{2s+1}} \(1- \frac{1}{2^{\frac12}} \sum_{r \text{ odd} } 2^{r\left(\frac{1}{2}-s\right)}\).
$$
Using $\frac{\zeta(2s-1)}{\zeta(2s)} = \sum_{c\ge 1} \frac{\phi(c^2)}{c^{2s+1}}$ and geometric summation we conclude the theorem.
\end{proof}
We let $s\to 1$ to obtain the evaluation of $Z_n(1)$.  By Theorem \ref{thm_E0} we have that
$$
Z_0(1) = \frac{4 e^{\frac{3 \pi i}{4}}}{3 \zeta(2)}
\lim_{s \to 1} \left(\zeta(2s-1) \left(1-2^{-(2s-1)} -2^{-s} \right) \right).
$$
Using that $\zeta(2s-1) = \frac{1}{2s-1} + O(1)$ as $s\to 1$, we obtain that
\[
\lim_{s \to 1}\zeta(2s-1)(1-2^{-(2s-1)}-2^{-s})=\frac{3}{4} \log(2).
\]
This  easily gives that $Z_0(1)=Z_0$.

\section{Relationship to other works}\label{sec:DIT}
Zagier \cite{Za2} (see also \cite{HZ}) showed that
  the generating function for the Hurwitz class numbers, namely
 $$
 \cH(\tau)
 := -\frac{1}{12} + \sum_{\begin{subarray}{c} n\ge 1 \\ n\equiv 0, 3\pmod{4} \end{subarray}} H(-n)q^n,
 $$
 is a mock modular form with shadow $\Theta$.
Recently, Duke, Imamo\=glu, and T\'oth  \cite{DIT} constructed a generalized mock modular form whose shadow is
the harmonic weak Maass form obtained by completing
$\cH(\tau)$ with a term similar to $P_g$ of the introduction.
One may use the function constructed in \cite{DIT} with the relations between
$L(\psi_{-n}, s)$, $r(n)$,  and the Hurwitz class numbers to give a different
construction of the form $F_{\Theta}$.
Their work does not include the explicit evaluation of the Fourier coefficients for square $n$, however one can use the
calculations here to compute those terms.

Much of the arithmetic of classical holomorphic
modular forms relies on the theory of complex multiplication and
the fact that the coefficients of modular forms are associated to the value of a modular function at points
determined by data associated to  an
imaginary quadratic fields.
The work of Duke, Imamo\=glu, and T\'oth \cite{DIT, DIT2}
demonstrates a similar phenomenon linking the coefficients of mock modular
forms and real quadratic fields.  Namely, they show that the coefficients of a family of mock modular forms
are associated to the values of modular functions at points corresponding to data from real quadratic fields.
Our result may be viewed as an additional example of this phenomenon.


Work of Bruinier and Ono \cite{BrO},
demonstrates the relationship between harmonic Maass forms and special values of derivatives of $L$-functions.
As in our work,
their work concerns twists by Dirichlet characters associated to both real and imaginary quadratic fields.
That work, as well as the related works of Bruinier, Kudla, and Yang \cite{BrKY, BrY1, BrY2}
yield deep connections between weak Maass forms  and
the theorems of Waldspurger,  Borcherds, and Gross-Zagier.

\end{document}